\documentclass{article}
\usepackage{amsmath}
\usepackage{amsthm}
\usepackage{amsfonts}
\usepackage{amssymb}
\usepackage{tikz}
\usepackage{pifont}
\usepackage{xcolor}
\usepackage{mathtools}
\usepackage[spaces,hyphens]{url}
\usepackage{hyperref}
\usepackage{authblk}
\usepackage[algo2e]{algorithm2e} %ALGO
\usepackage{algorithm,algpseudocode,float}

\theoremstyle{definition}
\newtheorem{exmp}{Example}[section]
\newtheorem{rmk}{Remark}[section]
\newtheorem{thm}{Theorem}[section]
\newtheorem{lemma}{Lemma}[section]
\usetikzlibrary{cd}

\newcommand{\Q}{\mathbb{Q}}

\newcommand{\C}{\mathbb{C}}

\newcommand{\cmark}{\ding{51}}

\newcommand{\xmark}{\ding{55}}
\newcommand{\gxmark}{\textcolor{lightgray}{\ding{55}}}

\theoremstyle{definition}

\makeatletter
\newenvironment{breakablealgorithm}
  {% \begin{breakablealgorithm}
     \refstepcounter{algorithm}% New algorithm
     \hrule height.8pt depth0pt \kern2pt% \@fs@pre for \@fs@ruled
     \renewcommand{\caption}[2][\relax]{% Make a new \caption
       {\raggedright\textbf{\ALG@name~\thealgorithm} ##2\par}%
       \ifx\relax##1\relax % #1 is \relax
         \addcontentsline{loa}{algorithm}{\protect\numberline{\thealgorithm}##2}%
       \else % #1 is not \relax
         \addcontentsline{loa}{algorithm}{\protect\numberline{\thealgorithm}##1}%
       \fi
       \kern2pt\hrule\kern2pt
     }
  }{% \end{breakablealgorithm}
     \kern2pt\hrule\relax% \@fs@post for \@fs@ruled
  }
\makeatother

\title{Orbit and Orbit Closure Containments for Cubic Surfaces}
\author[$\ddag$]{Eunice Sukarto$^{\star}$\\ Supervisors: Ralph Morrison$^{\dagger}$ and Bernd Sturmfels}

\affil[$\star$]{\emph{University of California, Berkeley}, \url{eunicesukarto@berkeley.edu}}
\affil[$\dagger$]{\emph{Williams College}, \url{10rem@williams.edu}}
\affil[$\ddag$]{\emph{MPI-MiS Leipzig and UC Berkeley}, \url{bernd@mis.mpg.de}}

\date{June 19, 2020}

\begin{document}

\maketitle

%Maybe make an abstract?
\begin{abstract}
Given two elements of a vector space acted on by a reductive group, we ask whether they lie in the same orbit, and if not, whether one lies in the orbit closure of the other. We develop techniques to optimize the orbit and orbit closure algorithms and apply these to give a partial classification of orbit closure containments in the case of cubic surfaces with infinitely many singular points, which are known to fall into 13 normal forms. We also discuss the computational obstructions to completing this classification, and discuss tools for future work in this direction.
\end{abstract}

\section{Introduction}
Let $G$ be a reductive group over an algebraically closed field $k$ acting on a $k$-vector space $V$. Throughout most of this paper we will take $G$ to be the general linear group $GL(n,\C)$ and $V$ to be $\mbox{Sym}_d(\C^n)$, the space of homogeneous degree $d$ polynomials in $n$ variables. Given $v\in V$, we look at its orbit $G\cdot v$ under $G$, and the Zariski closure of its orbit $\overline{G\cdot v}$ in $V$.
%Maybe  "its orbit $G\cdot v$ under $G$, and the orbit's Zariski closure $\overline{G\cdot v}$ in $V$"?  Otherwise thre are two "its"s, referring to different things -- FIXED
For each $w\in V$, we may ask the following questions:

\begin{itemize}
    \item[] Is $w\in G\cdot v$?
    \begin{itemize}
        \item[(1)] If yes, find $g$ such that $w = g\cdot v$.
        \item[(2)] If no, give a certificate. We then ask: is $w\in \overline{G\cdot v}$?
        \begin{itemize}
            \item[-] If yes, find a 1-parameter family in $G\cdot v$ which converges to $w$.
            \item[-] If no, find $f\in I(G\cdot v)$ such that $f(w)\neq 0$
        \end{itemize}
    \end{itemize}
\end{itemize}
%I think each of the above questions should start with a capitalized letter, since each is a complete sentence. --FIXED

From now on, unless otherwise stated, the rank of a polynomial will refer to its \emph{Waring rank}.
%probably need a comma here -- FIXED
This is the length of a shortest decomposition as a sum of powers of linear forms, the symmetric rank of the corresponding tensor~\cite{LanTei}.
%In general its good to put in references to textbooks, etc. so that the reader can know where to look if they want more info.  Maybe here you could refer to a textbook that talks about ranks of tensors? -- can't find a reference...

\begin{exmp}
(Strict containment $G\cdot v \subsetneq \overline{G\cdot v}$). Consider the action of $G = GL(4,\C)$ on $V = \mbox{Sym}_3(\C^4)$ as a change of coordinates $g\cdot v = v\circ g$. The rank 3 polynomial $w = x_1^2x_2$ (see~\cite[Section 10.4]{Lan} or~\cite[Section 8]{LanTei}) lies in the orbit closure of the rank 6 polynomial $v = x_1x_2^2 + x_3x_4^2$ ~\cite[\S 97]{Segre}. By Chevalley's Theorem~\cite[Corollary 4.20]{191},
%Citation
 the Zariski and Euclidean closures
%I think this should either be "the Zariski closure and the Euclidean closure", or "the Zariski and Euclidean closures" -- FIXED
coincide. Setting $x_1 = x_1$, $x_2 = x_2$, $x_3 = \epsilon x_3$, $x_4 = \epsilon x_4$, we get $v_{\epsilon} = x_1x_2^2 + \epsilon^3x_3x_4^2 \rightarrow w$ as $\epsilon \rightarrow 0$, but $w$ is not in the orbit of $v$ since their ranks do not agree.
%How do we know their ranks don't agree?
This gives an example of when the orbit is not closed.
\end{exmp}

Our paper is organized as follows. We look at standard algorithms for orbits and orbit closures in Section \ref{section:algo} and implement them in Macaulay2 in Section \ref{section:m2-code}. In Section \ref{section:tricks}, we develop some tricks to optimize these Macaulay2 implementations. In Section \ref{section:infinitely-singular}, we give a partial classification of orbit closure containments for cubic surfaces with infinitely many singular points and finally, we discuss computational obstructions to completing this classification and discuss tools for future work in this direction.

Throughout this paper, we will use the notations $(f_1, \ldots, f_n)$, $\left<f_1, \ldots, f_n\right>$, and ideal$(f_1, \ldots, f_n)$ to denote the ideal generated by $f_1, \ldots, f_n$.

\section{The Algorithms}\label{section:algo}
Recall that $G$ is a reductive group over an algebraically closed field $k$ acting on a vector space $V$, and $v,w$ are elements of $V$. In this section, we give standard algorithms for deciding whether $w\in G\cdot v$ and $w\in \overline{G\cdot v}$. 
\subsection{Is $w\in G\cdot v$?}
The following algorithm decides whether $w$ is in the orbit of $v$.

\bigskip
\begin{breakablealgorithm}
  \caption{Orbit containment}
    \KwIn{
    \begin{itemize}
        \item[-] a reductive group $G$ acting on a vector space $V$ over an algebraically closed field $k$ with basis $\{x_1,...,x_n\}$
        \item[-] $v, w\in V$
    \end{itemize}}
    \KwOut{
    \begin{itemize}
        \item[-] $J$, where $J =  \left\{ \begin{array}{rcl}
    (0) & \mbox{if} & w\in G\cdot v \\ 
    (1) & \mbox{if} & w\notin G\cdot v
    \end{array}\right.$
    \end{itemize}}
    \textbf{inOrbit}$(G,v,w)$ 
     \begin{itemize}
        \item[1.] Write $$g\cdot x_i = \sum_{j=1}^n g_{ij}x_j \text{ for each $i$},\ g = \left( \begin{array}{ccc} g_{11} & \dots & g_{1n} \\
        \vdots & \ddots \\
        g_{n1} & \dots & g_{nn} \end{array} \right)$$
        $$w = \sum_{i=1}^n w_ix_i$$
        $$v = \sum_{i=1}^n v_ix_i \Rightarrow g\cdot v = \sum_{i=1}^n v_i(\sum_{j=1}^n g_{ij}x_j) = \sum_{i=1}^n (\sum_{j=1}^n v_j g_{ji})x_i$$
        \item[2.] $I = \mbox{ideal}(\sum_{j=1}^n v_j g_{ji} - c_i | 1\leq i,j\leq n) + \mbox{ideal}(\det g - 1)$.
        \item[3.] Substitute the coefficients $c_i$ in $I$ by $w_i$.
        \item[4.] Eliminate the variables $g_{ij}$ to obtain the ideal $J$.
    \end{itemize}
 \label{algo:orbitContainment}
\end{breakablealgorithm}

\begin{proof}[Proof of correctness] Elements of $V$ can be written uniquely as $\sum_{i=1}^n c_ix_i$. Consider the polynomial map $G\rightarrow V$ given by $g\mapsto g\cdot v$. Its
graph is defined by the ideal $I = (\sum_{j=1}^n v_j g_{ji} - c_i | 1\leq i,j\leq n)$. 
%For the following sentence:  it's awkward to start a sentence with a mathematical symbol; is it possible to rephrase? -- FIXED
Since $w$ is in the orbit of $v$, there exists an invertible $g\in G$ such that $w = g\cdot v$. Equivalently, $w_i = \sum_{j=1}^n v_j g_{ji}$ for each $i$, i.e. $w$ is in the vanishing locus of the ideal $I$, so is in the set theoretic projection of $\mathcal{V}(I)$ onto $V$, where $\mathcal{V}(I)$ is the variety cut out by $I$. Substituting the coefficients $c_i$ in $I$ by $w_i$ sets up the system of equations, and eliminating the variables $g_{ij}$ solves for $g$. If $J = (1)$, then 1 is in the ideal obtained from $I$ by substituting the coefficients $c_i$ by $w_i$, so some combination of the equations results in an inconsistent equation, so a solution does not exist; otherwise $J = (0)$ and the equations are consistent, so there exists a solution $g$~\cite[Theorem 11.12, Derksen's Algorithm]{191}.
\end{proof}
%I don't know if there should be a different environment for correctness; probably just a \begin{proof} \end{proof} would suffice; you could do \begin{proof}[Proof of correctness of Algorithm 1]  \end{proof} if you want it to typset that as part of the proof heading.  One advantage of using the proof environment is you get the little box at the end, showing when the proof is over. --FIXED

% CONTEKAN
%\begin{itemize}
%    \item[1.] Write $$g\cdot x_i = \sum_{j=1}^n g_{ij}x_j \text{ for each $i$}$$
%     $$w = \sum_{i=1}^n w_ix_i$$
%     $$v = \sum_{i=1}^n v_ix_i \Rightarrow g\cdot v = \sum_{i=1}^n v_i(\sum_{j=1}^n g_{ij}x_j) = \sum_{i=1}^n (\sum_{j=1}^n v_j g_{ji})x_i$$
%     An element of $V$ can be written uniquely as $\sum_{i=1}^n c_ix_i$.
%     \item[2.] Consider the polynomial map $G\rightarrow V$ given by $g\mapsto g\cdot v$. It's graph is defined by the ideal $I = (\sum_{j=1}^n v_j g_{ji} - c_i | 1\leq i,j\leq n)$.
%     \item[3.] $w$ is in the orbit of $v$ iff it is in the set theoretic projection of $\mathcal{V}(I)$ onto $V$. To compute the image of the set theoretic projection, substitute the $c_i$s in $I$ by $w_i$s, then eliminate the $g_{ij}$s to obtain $J$.
% \end{itemize}

\begin{exmp} 
($G = GL(4,\C)$, $V = Sym_3(\C^4)$). Let 
\begin{align*}
v &= x_1x_3x_4 + x_3^3 \\
w &= 125x_2^3 + 5x_1x_2x_3 + 525x_2^2x_3 + 7x_1x_3^2 + 745x_2x_3^2 + 357x_3^3 + 5x_1x_2x_4 +\\
& 75x_2^2x_4 + 8x_1x_3x_4 + 220x_2x_3x_4 + 163x_3^2x_4 + x_1x_4^2 + 15x_2x_4^2 + 23x_3x_4^2 + x_4^3
\end{align*}
where $w$ is obtained from $v$ by setting $x_1 := x_1 + 2x_3$, $x_2 := x_2$, $x_3 := 5x_2 + 7x_3 + x_4$, $x_4 = x_3 + x_4$. The following Macaulay2 code computes $J$~\cite{M2}.

\begin{footnotesize}
\begin{verbatim}
R = QQ[a1,a2,a3,a4,b1,b2,b3,b4,c1,c2,c3,c4,d1,d2,d3,d4][x1,x2,x3,x4];
M = matrix{{a1,a2,a3,a4}, {b1,b2,b3,b4}, {c1,c2,c3,c4}, {d1,d2,d3,d4}};
f1 = 125*x2^3 + 5*x1*x2*x3 + 525*x2^2*x3 + 7*x1*x3^2 + 745*x2*x3^2 + 357*x3^3 
      + 5*x1*x2*x4 + 75*x2^2*x4 + 8*x1*x3*x4 + 220*x2*x3*x4 + 163*x3^2*x4 
      + x1*x4^2 + 15*x2*x4^2 + 23*x3*x4^2 + x4^3;
l1 = a1*x1 + a2*x2 + a3*x3 + a4*x4;
l2 = b1*x1 + b2*x2 + b3*x3 + b4*x4;
l3 = c1*x1 + c2*x2 + c3*x3 + c4*x4;
l4 = d1*x1 + d2*x2 + d3*x3 + d4*x4;
f2 = l1*l3*l4 + l3^3;
g = f1 - f2;
(N,C) = coefficients g;
I = minors(1,C) + ideal(det M - 1);
S = QQ[a1,a2,a3,a4,b1,b2,b3,b4,c1,c2,c3,c4,d1,d2,d3,d4];
I = substitute(I,S);
time J = eliminate(I,{a1,a2,a3,a4,b1,b2,b3,b4,c1,c2,c3,c4,d1,d2,d3,d4});
\end{verbatim}
\end{footnotesize}
This computation gives $J = (0)$ as expected since $w$ is in the orbit of $v$. If we set $w = x_1^3 + x_2^3 + x_3^3$ instead, we get $J = (1)$, which is again expected since $w$ is not in the orbit of $v$ since their singular loci are not projectively equivalent (see Section \ref{singularities}).
\end{exmp}

\subsection{Is $w\in \overline{G\cdot v}$?}
The following algorithm decides whether $w$ is in the orbit closure of $v$.

\bigskip
\begin{breakablealgorithm}
\caption{Orbit closure containment}
\KwIn{
\begin{itemize}
    \item[-] a reductive group $G$ acting on a vector space $V$ over an algebraically closed field $k$ with basis $\{x_1,...,x_n\}$
    \item[-] $v, w\in V$
\end{itemize}}
\KwOut{
\begin{itemize}
    \item[-] $K$, where $K =  \left\{ \begin{array}{rcl}
(0) & \mbox{if} & w\in \overline{G\cdot v} \\ 
(1) & \mbox{if} & w\notin \overline{G\cdot v}
\end{array}\right.$
\end{itemize}}
\textbf{inOrbitClosure}$(G,v,w)$ 
 \begin{itemize}
    \item[1.] Same as Step 1 of Algorithm \ref{algo:orbitContainment}: Orbit Containment.
    \item[2.] $I = \mbox{ideal}(\sum_{j=1}^n v_j g_{ji} - c_i | 1\leq i,j\leq n)$.
    \item[3.] Eliminate the variables $g_{ij}$ to obtain the ideal $J$.
    \item[4.] Substitute the coefficients $c_i$ in $J$ by $w_i$ to obtain the ideal $K$.
\end{itemize}
 \label{algo:orbitClosureContainment}
\end{breakablealgorithm}

\begin{proof}[Proof of correctness] With notation as in Algorithm \ref{algo:orbitContainment}, the orbit of $v$ is given by the set theoretic projection $\pi$ of $\mathcal{V}(I)$ onto $V$. The orbit closure of $v$ is its Zariski (equals to Euclidean by Chevalley's Theorem)
%write out in words, i.e. "equal to the Euclidean closure by etc")
closure, which by~\cite[Theorem 4.2]{191} is given by the elimination ideal $J = I\cap k[c_1, \cdots, c_n]$. Substituting the coefficients $c_i$ by $w_i$ tells whether $w$ is the zero locus of $J$~\cite[Theorem 11.12, Derksen's Algorithm]{191}. 

Note that here we do not impose the condition det $G = 1$ since any matrix (invertible or not) is in the Euclidean closure of the set of invertible matrices. Indeed, if $w = \lim_{\epsilon \to 0} g_n\cdot v_n$ where $g_n$ is not necessarily invertible, since each $g_n = \lim_{\epsilon \to 0}g_{nm}$ where $g_{nm}$ is invertible, then $w = \lim_{\epsilon \to 0} g_{nm_n} \cdot v_n$ (for some $m_n$) can be written as a limit using only invertible matrices. 
\end{proof}

% CONTEKAN
% Steps:
% \begin{itemize}
%     \item[1.] Same as in orbit containment.
%     \item[2.] Same as in orbit containment.
%     \item[3.] $w$ is in the orbit closure of $v$ iff it is in the (Zariski closure of the set theoretic) projection of $\mathcal{V}(I)$ onto $V$. Eliminating the $g_{ij}$s from $I$ gives us the ideal $J$ which cuts out the image of the projection. Substituting the $c_i$s in $J$ by the $w_i$s gives us $K$.
% \end{itemize}

\begin{exmp}
($G = GL(2,\C)$, $V = \mbox{Sym}_3(\C^2)$). Let 
\begin{align*}
v &=  x_1^3 + x_1x_2^2\\
w &= x_1^2x_2
\end{align*}
The following Macaulay2 code computes $K$.

\begin{footnotesize}
\begin{verbatim}
R = QQ[c30,c21,c12,c03,a1,a2,b1,b2][x1,x2,x3,x4];
f1 = c30*x1^3 + c21*x1^2*x2 + c12*x1*x2^2 + c03*x2^3;
l1 = a1*x1 + a2*x2
l2 = b1*x1 + b2*x2
f2 = l1^3 + l1*l2^2;
g = f1 - f2;
(M,C) = coefficients g;
I = minors(1,C);
S = QQ[c30,c21,c12,c03,a1,a2,b1,b2];
I = substitute(I,S);
time J = eliminate(I,{a1,a2,b1,b2});
K = substitute(J,{c30=>0,c21=>1,c12=>0,c03=>0})
\end{verbatim}
\end{footnotesize}
This computation gives $K = (0)$, which means $w$ is in the orbit closure of $v$. To verify this, since $$x_1 x_2 = \lim_{\epsilon \to 0} \left( \frac{x_1^2 - (x_1 - \epsilon x_2)^2}{2\epsilon} \right),$$
$x_1 x_2$ can be written as a limit of polynomials of the form $y_1^2 + y_2^2$, hence we can write $x_1^2x_2$ as a limit of polynomials of the form $y_1^3 + y_1y_2^2$. Since we are working over $\C$, by Chevalley's Theorem, the Euclidean and Zariski closures of the image of the projection coincide, so $w\in \overline{G\cdot v}$. Reversing the roles of $v$ and $w$ gives $K = (1)$, so $v\notin \overline{G\cdot w}$ as expected since the stabilizer of $w$ has codimension lower than that of $v$ (see Section \ref{dimension}). Thus,
$w$ is in the orbit closure, but not in the orbit of $v$.  %should that be "the orbit of v"? -- ah right
\end{exmp}

\section{A Macaulay2 Implementation for Cubic Surfaces}\label{section:m2-code}
In this section, we provide Macaulay2 implementations of the orbit and orbit closure algorithms for cubic surfaces, homogeneous degree 3 polynomials in 4 variables with complex coefficients, written $$v = \sum_{i+j+k+l = 3} c_{ijkl}x_1^ix_2^jx_3^kx_4^l.$$ The projective linear group $G = PGL(4, \C)$ acts on a cubic surface $v$ by $g\cdot v = v\circ g$ for each $g\in G$. We restrict ourselves to inputs with rational coefficients.

\subsection{Orbit} \label{orbit_algo}
The following Macaulay2 code computes whether $f_1$ is in the orbit of $f_2$: $f_1$ is in the orbit closure of $f_2$ iff $J = (0)$.
\begin{footnotesize}
\begin{verbatim}
R = QQ[a1,a2,a3,a4,b1,b2,b3,b4,c1,c2,c3,c4,d1,d2,d3,d4][x1,x2,x3,x4];
M = matrix{{a1,a2,a3,a4}, {b1,b2,b3,b4}, {c1,c2,c3,c4}, {d1,d2,d3,d4}};
f1 = -- plug in f1 here
l1 = a1*x1 + a2*x2 + a3*x3 + a4*x4;
l2 = b1*x1 + b2*x2 + b3*x3 + b4*x4;
l3 = c1*x1 + c2*x2 + c3*x3 + c4*x4;
l4 = d1*x1 + d2*x2 + d3*x3 + d4*x4;
f2 = -- plug in f2 here, replacing x1 by l1, etc.
g = f1 - f2;
(M,C) = coefficients g;
I = minors(1,C);
I = I + (det M - 1);
S = QQ[a1,a2,a3,a4,b1,b2,b3,b4,c1,c2,c3,c4,d1,d2,d3,d4];
I = substitute(I,S);
time J = eliminate(I,{a1,a2,a3,a4,b1,b2,b3,b4,c1,c2,c3,c4,d1,d2,d3,d4});
\end{verbatim}
\end{footnotesize}

\subsection{Orbit Closure} \label{orbit_closure_algo}
The following Macaulay2 code computes whether $f$ is in the orbit closure of $f_2$: $f$ is in the orbit closure of $f_2$ iff $K = (0)$.
\begin{footnotesize}
\begin{verbatim}
R = QQ[c3000,c2100,c1200,c0300,c2010,c1110,c0210,c1020,c0120,c0030,
       c2001,c1101,c1011,c0201,c1002,c0102,c0003,c0021,c0012,c0111,
       a1,a2,a3,a4,b1,b2,b3,b4,c1,c2,c3,c4,d1,d2,d3,d4][x1,x2,x3,x4];
M = matrix {{x1^3, x1^2*x2, x1*x2^2, x2^3, x1^2*x3, x1*x2*x3, x2^2*x3, x1*x3^2, 
       x2*x3^2, x3^3,x1^2*x4,x1*x2*x4,x1*x3*x4,x2^2*x4,x1*x4^2,x2*x4^2,x4^3,
       x3^2*x1,x3*x4^2,x2*x3*x4}}
C = matrix{{c3000,c2100,c1200,c0300,c2010,c1110,c0210,c1020,c0120,c0030,
       c2001,c1101,c1011,c0201,c1002,c0102,c0003,c0021,c0012,c0111}};
f1 = M*transpose(C);
f1 = f1_(0,0);
l1 = a1*x1 + a2*x2 + a3*x3 + a4*x4;
l2 = b1*x1 + b2*x2 + b3*x3 + b4*x4;
l3 = c1*x1 + c2*x2 + c3*x3 + c4*x4;
l4 = d1*x1 + d2*x2 + d3*x3 + d4*x4;
f2 = -- plug in f2 here, replacing x1 by l1, etc.
g = f1 - f2;
(M,C) = coefficients g;
I = minors(1,C);
S = QQ[c3000,c2100,c1200,c0300,c2010,c1110,c0210,c1020,c0120,c0030,
       c2001,c1101,c1011,c0201,c1002,c0102,c0003,c0021,c0012,c0111,
       a1,a2,a3,a4,b1,b2,b3,b4,c1,c2,c3,c4,d1,d2,d3,d4];
I = substitute(I,S);
time J = eliminate(I,{a1,a2,a3,a4,b1,b2,b3,b4,c1,c2,c3,c4,d1,d2,d3,d4});
-- substitute the values of the coefficients by the coefficients of f 
       (replace the 0s)
K = substitute(J,{c3000=>0,c2100=>0,c1200=>0,c0300=>0,c2010=>0,c1110=>0,c0210=>0,
       c1020=>0,c0120=>0,c2001=>0,c1101=>0,c0201=>0,c1002=>0,c0102=>0,c0030=>0,
       c0003=>0,c0021=>0,c0012=>0,c0111=>0})
\end{verbatim}
\end{footnotesize}

\begin{rmk}\label{rmk:justification-m2}
Here we justify the use of Macaulay2 for computations on inputs with rational coefficients. Indeed, the only nontrivial Macaulay2 routines used are \emph{substitute} and \emph{eliminate}. 

By construction, the ideal $I$ has rational generators. The ideal obtained by specializing $I$ is generated by the specialization of each of the rational generators of $I$. Since the values we are substituting are rationals, substitution over the rationals is the same as substitution over the complex field. 

For elimination, we refer to~\cite[Theorems 4.2, 4.5]{191}, restated here as Theorem \ref{thm:math191} for convenience.
\begin{thm}\label{thm:math191}
Let $\pi: k^n \to k^m$ be the projection $(a_1,\ldots,a_n) \mapsto (a_1,\ldots,a_m)$
%I would recommend using (a_1,...,a_n) instead of (x_1,...,x_n), since you're using the x_i's for variables. -- FIXED
%Also, instead of "..." it looks better as \ldots -- FIXED
onto the subspace $k^m$. Let $I\subset k[x_1,...,x_n]$ be an ideal and $V = \mathcal{V}(I)$ its variety in $k^n$, where $k$ is an algebraically closed field. Then its closed image in $k^m$ is the variety $\overline{\pi(V)} = \mathcal{V}(J)$ defined by the elimination ideal $J = I \cap k[x_1,...,x_m]$.

Moreover, if $\mathcal{G}$ is a lexicographic Grobner basis for $I$ in $k[x_1,...,x_n]$, then $J$ has as Grobner basis $\mathcal{G}' = \mathcal{G}\cap k[x_1,...,x_m]$.
\end{thm}
Since $\C$ is algebraically closed, the elimination ideal cuts out the image of the Zariski closure of the projection of a variety. We want to show that the elimination ideal over $\Q$ is generated by the same generators as the elimination ideal over $\C$. The Macaulay2 implementation of \emph{eliminate} involves computing a Grobner basis and intersecting it with $k[x_1,...,x_m]$. It suffices to show that if $I$ has a set of generators consisting of rational polynomials, then it has a Grobner basis consisting of rational polynomials. Since Buchberger's algorithm, which is used in Macaulay2 to compute Grobner bases, only involves rational operations, the resulting Grobner bases must consist of rational polynomials. So, this set of polynomials generate the elimination ideal over $\Q$ and over $\C$, justifying our use of \emph{eliminate}.

Substitution clearly terminates. Elimination terminates theoretically since Buchberger's algorithm terminates.
\end{rmk}

\section{Some Strategies for Reducing Computation Time}\label{section:tricks}

In both the algorithms for computing orbits and orbit closures, elimination is the only computationally expensive routine.
%Isn't saturation also expensive? 
Here we compile a list of tricks that sometimes helps in obtaining a possibly non-constructive yes/no answer to the orbit and orbit closure problems.

We give some examples to show the significance of elimination in determining whether or not a computation terminates. The following code which shows that $x_1x_3x_4 + x_3^3$ is not in the orbit closure of $x_3x_4^2$, terminates quickly.

\begin{footnotesize}
\begin{verbatim}
R = QQ[c300,c210,c120,c030,c201,c111,c021,c102,c012,c003,
    a1,a3,a4,b1,b3,b4][x1,x3,x4];
M = matrix {{x1^3, x1^2*x3, x1*x3^2, x3^3, x1^2*x4, x1*x3*x4, x3^2*x4,
      x1*x4^2, x3*x4^2, x4^3}}
C = matrix{{c300,c210,c120,c030,c201,c111,c021,c102,c012,c003}};
f1 = M*transpose(C);
f1 = f1_(0,0);
l3 = a1*x1 + a3*x3 + a4*x4;
l4 = b1*x1 + b3*x3 + b4*x4;
f2 = l3^2*l4;
g = f1 - f2;
(M,C) = coefficients g;
I = minors(1,C);
S = QQ[c300,c210,c120,c030,c201,c111,c021,c102,c012,c012,c003,
    a1,a3,a4,b1,b3,b4];
I = substitute(I,S);
time J = eliminate(I,{a1,a3,a4,b1,b3,b4});
K = substitute(J,{c300=>0,c210=>0,c120=>0,c201=>0,
c021=>0,c102=>0,c012=>0,c003=>0})
\end{verbatim}
\end{footnotesize}

However, the following code which tells whether $x_1^2x_2 + x_1x_3x_4 + x_3^3$ is in the orbit closure of $x_1x_2^2 + x_3x_4^2$, was run for a day, got stuck on elimination, and did not terminate. 

\begin{footnotesize}
\begin{verbatim}
R = QQ[c3000,c2100,c1200,c0300,c2010,c1110,c0210,c1020,c0120,c0030,
       c2001,c1101,c1011,c0201,c1002,c0102,c0003,c0021,c0012,c0111,
       a1,a2,a3,a4,b1,b3,c1,c2,c3,c4,d1,d3][x1,x2,x3,x4];
M = matrix {{x1^3, x1^2*x2, x1*x2^2, x2^3, x1^2*x3, x1*x2*x3, x2^2*x3,x1*x3^2, 
       x2*x3^2, x3^3, x1^2*x4,x1*x2*x4,x1*x3*x4,x2^2*x4,x1*x4^2,x2*x4^2,x4^3,
       x3^2*x4,x3*x4^2,x2*x3*x4}}
C = matrix{{c3000,c2100,c1200,c0300,c2010,c1110,c0210,c1020,c0120,c0030,
       c2001,c1101,c1011,c0201,c1002,c0102,c0003,c0021,c0012,c0111}};
f1 = M*transpose(C);
f1 = f1_(0,0);
l1 = c1*x1 + c2*x2 + c3*x3 + c4*x4;
l2 = d1*x1 + d3*x3;
l3 = a1*x1 + a2*x2 + a3*x3 + a4*x4;
l4 = b1*x1 + b3*x3;
f2 = l2^2*l1 + l4^2*l3;
g = f1 - f2;
(M,C) = coefficients g;
I = minors(1,C);
S = QQ[c3000,c2100,c1200,c0300,c2010,c1110,c0210,c1020,c0120,c0030,
       c2001,c1101,c1011,c0201,c1002,c0102,c0003,c0021,c0012,c0111,
       a1,a2,a3,a4,b1,b3,c1,c2,c3,c4,d1,d3];
I = substitute(I,S);
time J = eliminate(I,{a1,a2,a3,a4,b1,b3,c1,c2,c3,c4,d1,d3}); -- proj to 
       polynomials only involving c****'s
K = substitute(J,{c3000=>0,c1200=>0,c0300=>0,c2010=>0,c1110=>0,c0210=>0,c1020=>0,
       c0120=>0,c2001=>0,c1101=>0,c0201=>0,c1002=>0,c0102=>0,c0003=>0,c0021=>0,
       c0012=>0,c0111=>0,c2100=>1,c1011=>1,c0030=>1})
\end{verbatim}
\end{footnotesize}
%Typesetting:  the time J line (and some others) are spilling over the margins.  Any way to fix that? -- FIXED

We remark that the singular loci of $x_1^2x_2 + x_1x_3x_4 + x_3^3$ and $x_1x_2^2 + x_3x_4^2$ are the lines $[0:x_2:0:x_4]$ and $[x_1:0:x_3:0]$ respectively. If $x_1^2x_2 + x_1x_3x_4 + x_3^3$ is the limit of $\{f_n\}$ where $f_n\in G\cdot (x_1x_2^2 + x_3x_4^2)$, then the singular locus of $x_1^2x_2 + x_1x_3x_4 + x_3^3$ must be the limit of the singular loci of the $f_n$s. We will show in Example \ref{exmp:sub-elim-sub-example} that in this particular case, it suffices to consider elements of $G$ such that the second and forth coordinates are contained in the span of $x_1, x_3$. 

\subsection{Sub-elim-sub}\label{sub-elim-sub}
\begin{lemma}\label{lemma:sub-elim-sub}
Let $I$ be an ideal in $k[\vec{x}, \vec{y}, \vec{z}, \vec{w}]$ where $k$ is a field and $\vec{x} = (x_1,...,x_m)$, $\vec{y} = (y_1,...,y_n)$, $\vec{z} = (z_1,...,z_l)$, $\vec{w} = (w_1,...,w_r)$. Let $K$ be the ideal obtained by first eliminating $\vec{z}$, then substituting $\vec{x} = \vec{a}$ and $\vec{y} = \vec{b}$. Let $M$ be the ideal obtained by first substituting $\vec{x} = \vec{a}$, eliminating $\vec{z}$, then substituting $\vec{y} = \vec{b}$. \\
Then, $K\subset M$.
\end{lemma}
Informally, this is saying that the ideal obtained by eliminating then substituting is contained in the ideal obtained by first substituting some variables, then eliminating, then substituting the remaining variables. Since $M$ has fewer variables than $K$, performing elimination on $M$ is easier. If $M = 0$, then we know $K = 0$, but if $M \neq 0$, we cannot say anything about $K$. This sometimes helps in the orbit closure problem since elimination is done before substitution, but not in the orbit problem where elimination is already performed last. However, it takes a lot of guessing and trial and error to determine the order of substitution and elimination which would result in a terminating elimination with $M$ as small as possible.

\begin{proof}[Proof of Lemma \ref{lemma:sub-elim-sub}]
Let $I = \left<r_1,...,r_s\right>$. Let the elimination ideal $\mbox{Elim}(I,\vec{z}) = I\cap k[\vec{x}, \vec{y}, \vec{w}]$ be generated by $t_1,...,t_u$. So, $$K = \left<t_1(\vec{a},\vec{b},\vec{w}),...,t_u(\vec{a},\vec{b},\vec{w})\right>.$$ If $f\in K$, write $f = \sum p_it_i(\vec{a},\vec{b},\vec{w})$ where $p_i\in k[\vec{w}]$.  The ideal obtained from $I$ after substituting $\vec{x} = \vec{a}$ is $$\mbox{Sub}(I,\vec{x}) = \left<r_1(\vec{a},\vec{y},\vec{z},\vec{w}),..., r_s(\vec{a},\vec{y},\vec{z},\vec{w})\right>.$$
Since $t_1,...,t_u\in \mbox{Elim}(I,\vec{z}) = I\cap k[\vec{x}, \vec{y}, \vec{w}]$,
$$t_1(\vec{a},\vec{y},\vec{w}),..., t_u(\vec{a},\vec{y},\vec{w})\in k[\vec{y}, \vec{w}]$$
$$\mbox{and } t_1(\vec{a},\vec{y},\vec{w}),..., t_u(\vec{a},\vec{y},\vec{w})\in \mbox{Sub}(I,\vec{x}),$$
so $$t_1(\vec{a},\vec{y},\vec{w}),..., t_u(\vec{a},\vec{y},\vec{w})\in \mbox{Sub}(I,\vec{x})\cap k[\vec{y},\vec{w}] = \mbox{Elim}(\mbox{Sub}(I,\vec{x}), \vec{z}),$$ so $$t_1(\vec{a},\vec{b},\vec{w}),..., t_u(\vec{a},\vec{b},\vec{w})\in \mbox{Sub}(\mbox{Elim}(\mbox{Sub}(I,\vec{x}), \vec{z}), \vec{y}) = M.$$
Thus, $f\in M$.
\end{proof}

The reverse containment does not hold in general. For example, consider $I = \left<y+xz\right>$ in $k[x,y,z]$. Eliminating $z$ from $I$ gives the zero ideal, so the next substitutions do not matter and $K=(0)$. On the other hand, substituting $x=0$ to $I$ gives the ideal $\left<y\right>$. Eliminating $z$ from this ideal has no effect, and substituting $y=1$ results in $M=(1)\nsubseteq K$.\\

The following example shows how to use \emph{sub-elim-sub} to solve the above computation in determining whether $x_1^2x_2 + x_1x_3x_4 + x_3^3$ is in the orbit closure of $x_1x_2^2 + x_3x_4^2$.

\begin{exmp} \label{exmp:sub-elim-sub-example}
The claim allows us to reorder the substitution and elimination process as follows, resulting in a terminating algorithm.

\begin{footnotesize}
\begin{verbatim}
R = QQ[c3000,c2100,c1200,c0300,c2010,c1110,c0210,c1020,c0120,c0030,
       c2001,c1101,c1011,c0201,c1002,c0102,c0003,c0021,c0012,c0111,
       a1,a2,a3,a4,b1,b3,c1,c2,c3,c4,d1,d3][x1,x2,x3,x4];
M = matrix {{x1^3, x1^2*x2, x1*x2^2, x2^3, x1^2*x3, x1*x2*x3, x2^2*x3,x1*x3^2, 
            x2*x3^2, x3^3, x1^2*x4, x1*x2*x4, x1*x3*x4, x2^2*x4, x1*x4^2, x2*x4^2, 
            x4^3, x3^2*x4, x3*x4^2, x2*x3*x4}}
C = matrix{{c3000,c2100,c1200,c0300,c2010,c1110,c0210,c1020,c0120,c0030,c2001,
            c1101,c1011,c0201,c1002,c0102,c0003,c0021,c0012,c0111}};
f1 = M*transpose(C);
f1 = f1_(0,0);
l1 = c1*x1 + c2*x2 + c3*x3 + c4*x4;
l2 = d1*x1 + d3*x3;
l3 = a1*x1 + a2*x2 + a3*x3 + a4*x4;
l4 = b1*x1 + b3*x3;
f2 = l2^2*l1 + l4^2*l3;
g = f1 - f2;
(M,C) = coefficients g;
I = minors(1,C);
S = QQ[c3000,c2100,c1200,c0300,c2010,c1110,c0210,c1020,c0120,c0030,
       c2001,c1101,c1011,c0201,c1002,c0102,c0003,c0021,c0012,c0111,
       a1,a2,a3,a4,b1,b3,c1,c2,c3,c4,d1,d3];
I = substitute(I,S);
J = substitute(I,{c1200=>0,c0300=>0,c0210=>0,c1101=>0,c0201=>0,c0111=>0,c1002=>0,
    c0102=>0,c0012=>0,c0003=>0,c3000=>0,c2010=>0,c1110=>0,c1020=>0,c0021=>0});
time J = eliminate(J,{a1,a2,a3,a4,b1,b3,c1,c2,c3,c4,d1,d3});
K = substitute(J,{c0120=>0,c2001=>0,c2100=>1,c1011=>1,c0030=>1})
\end{verbatim}
\end{footnotesize}
The last subsitution is actually not necessary since $J$ is already $(0)$. Applying the claim to $M_{\text{claim}}=K$ and $K_{\text{claim}}=K_{\text{bad}}=K$ from the previous, non-terminating code, we get $K_{\text{bad}}\subset K = (0)$, so $K_{\text{bad}} = (0)$, showing that $x_1^2x_2 + x_1x_3x_4 + x_3^3$ is in the orbit closure of $x_1x_2^2 + x_3x_4^2$. Note that since $x_1^2x_2 + x_1x_3x_4 + x_3^3$ is in the orbit closure of $x_1x_2^2 + x_3x_4^2$ under elements of $G$ such that the second and fourth coordinates are contained in the span of $x_1, x_3$, it is also contained in the orbit closure of $x_1x_2^2 + x_3x_4^2$ under all elements of $G$. 
\end{exmp}

\subsection{Looking at Singularities}\label{singularities}
A singular point of a homogenous degree $d$ polynomial $f$ in $n$ variables is a point $p\in \mathcal{V}(f)$ at which all partial derivatives of $f$ vanish. It is a point at which the tangent space $T_p\mathcal{V}(f)$ has dimension higher than that of the variety $\mathcal{V}(f)$. 

If $\nabla f = 0$, then $\nabla (g\cdot f) = \nabla (f\circ g) = (\nabla f) \circ Dg = (\nabla f) \circ g = g\cdot \nabla f$ where the second to last equality is due to the fact that a linear function is its own derivative. So, if $g\cdot v = w$, then $g$ must map the singular locus of $v$ to the singular locus of $w$. In particular, if the singular loci of $v$ and $w$ are not projectively equivalent, then $v$ and $w$ do not lie in the same orbit.

In the case of cubic surfaces, the singular locus is cut out by 4 homogeneous degree 2 polynomials. By orthonormalization, every symmetric bilinear form  can be transformed into the form $\epsilon x_1^2 + \dots + \epsilon x_n^2$ where $\epsilon \in \{0, 1\}$, which helps in seeing whether the singular loci are projectively equivalent.

\subsection{Counting Dimensions}\label{dimension}
Clearly, if the dimensions $\dim (G\cdot v)\neq \dim (G\cdot w)$, then the orbits $G\cdot v$ and $G\cdot w$ are distinct, and thus disjoint. For orbit closures, recall that we are working over $\C$. This allows us to use Chevalley's Theorem which says that the Zariski and Euclidean closures coincide. 
\begin{lemma}
Let $G, v, w$ be as before, $G\cdot v \neq G\cdot w$.
\begin{itemize}
    \item [(i)] If $w\in \overline{G\cdot v}$, then $\overline{G\cdot w} \subset \overline{G\cdot v}$.
    \item [(ii)] If $\dim \overline{G\cdot v} \leq \dim \overline{G\cdot w}$, then $w\notin \overline{G\cdot v}$.
\end{itemize}
\end{lemma}
\begin{proof}
To show (i), let $\{g_n\cdot v\}\longrightarrow w$. Then for each $g\in G$, $\{gg_n\cdot v\}\longrightarrow g\cdot w$, so $G\cdot w\subset \overline{G\cdot v}$. Since $\overline{G\cdot v}$ is closed, it contains the closure $\overline{G\cdot w}$ of $G\cdot w$.

For (ii), if $\dim \overline{G\cdot v} < \dim \overline{G\cdot w}$, then $\overline{G\cdot v} \not \supset \overline{G\cdot w}$ since a manifold cannot contain a submanifold of higher dimension, thus $w\notin \overline{G\cdot v}$ by $(i)$.

If $\dim \overline{G\cdot v} = \dim \overline{G\cdot w}$, suppose by contradiction that $w\in \overline{G\cdot v}$, so $\overline{G\cdot w} \subset \overline{G\cdot v}$. By Chevalley's Theorem, the image of a polynomial map is a constructible set, a finite union of differences of varieties. Write $G\cdot v = (U_1-V_1)\cup \dots \cup (U_k-V_k) = (U_1-U_1\cap V_1)\cup \dots \cup (U_k-U_k\cap V_k)$ where each $U_i-V_i\neq \emptyset$. By decomposing $U_1,\dots, U_k$ into their irreducible components, we may assume $U_1,\dots, U_k$ are irreducible. Since $U_i\cap V_i$ is a proper closed subset, $U_i - U_i\cap V_i$ is an open dense subset of $U_i$, so $\overline{G\cdot v} = \bigcup_{i=1}^k \overline{U_i}$. Moreover, $\dim U_i\cap V_i\leq \dim U_i - 1$ since $U_i\cap V_i$ is a proper subvariety of the irreducible variety $U_i$, so has empty interior. Since $G\cdot v \neq G\cdot w$, $G\cdot v \cap G\cdot w = \emptyset$, so $\overline{G\cdot w}\subset \overline{G\cdot v} - G\cdot v \subset \bigcup_{i=1}^k (U_i\cap V_i)$ thus $\dim \overline{G\cdot w} \leq \dim \bigcup_{i=1}^k (U_i\cap V_i) \leq \mbox{max}_{i=1}^k (\dim U_i\cap V_i) \leq \mbox{max}_{i=1}^k (\dim U_i) - 1\leq \dim \overline{G\cdot v} - 1 < \dim \overline{G\cdot w}$, a contradiction, hence proving the equality case of $(ii)$.
\end{proof}

We now provide Macaulay2 code to compute the dimension of the orbit closure of cubic surfaces over $\C$. Here, $I$ is the stabilizer of $f$, so by~\cite[Proposition 21.4.3]{TauYu}, the codimension of $I$ is the dimension of the orbit of $f$.

\begin{footnotesize}
\begin{verbatim}
    R = QQ[m_(1,1)..m_(4,4)][x_1..x_4];
    M = mutableMatrix(R,4,4);
    for i to 3 do ( for j to 3 do (
    M_(i,j) = m_(i+1,j+1);
    ); );
    M = matrix M;
    for i from 1 to 4 do (
        y_i = 0;
        for j from 1 to 4 do (
    y_i = y_i + m_(i,j)*x_j;
    );
        );
    f = -- plug in f here with variables x1, ..., x4
    g = substitute(f,{x_1=>y_1,x_2=>y_2,x_3=>y_3,x_4=>y_4});
    F = f-g;
    (Mon,C) = coefficients F;
    I = minors(1,C);
    S = QQ[m_(1,1)..m_(4,4)];
    I = substitute(I,S);
    codim I
\end{verbatim}
\end{footnotesize}

% \subsection{Finite Fields}
% //TODO
% \begin{claim}
% want to compute orbit closure of f
% \end{claim}
% \begin{proof}
% Clearing denominators and dividing out by common factors, we may assume that $f$ has integer coefficients. Let $f$ be a nonzero polynomial in $\mbox{Elim}(I, \vec{y}) = I\cap \C[\vec{x}]$ with rational coefficients. Since the coefficients of $f$ have no common factors, $[f] \neq 0$ in $\Z/p\Z[\vec{x}]$. Thus, $[f]$ is a nonzero element in $I\mod p$ \textcolor{red}{//TODO continue laterlah}
% \end{proof}

\section{Cubic Surfaces with Infinitely Many Singular Points}\label{section:infinitely-singular}
In this section, we apply some of the techniques in Section \ref{section:tricks} to compute orbit closure containments for cubic surfaces with infinitely many singular points under the action of the projective linear group $PGL(4,\C)$. From~\cite{CGV,Sei},~\cite[\S 97]{Segre} and~\cite[Section 8]{LanTei}, there are only finitely many normal forms of such cubic surfaces, given in Table \ref{table:infinite_singularities_normal_form}. Note that these normal forms are completely characterized by their ranks and singularities, so each normal form corresponds to a distinct orbit.

\begin{table}[ht]
\centering
\begin{tabular}{c|c|c|c|c|c}
No & Label & Rank & Normal form & Singularity & Dim \\ \hline
1. & 1A & 1 & $x_1^3$ & plane $[0:x_2:x_3:x_4]$ & 4 \\ 
2. & 2A & 2 & $x_1^3 + x_2^3$ & line $[0:0:x_3:x_4]$ & 8 \\
3. & 3A & 3 & $x_1^2 x_2$ & plane $[0:x_2:x_3:x_4]$ & 7 \\
4. & 4A & 4 & $x_2^2 x_3 - x_1^3 - x_1^2 x_3$ & line $[0:0:x_3:x_4]$ & 12 \\
5. & 4B & 4 & $x_2^2 x_3 - x_1^3$ & line $[0:0:x_3:x_4]$ & 11 \\
6. & 4C & 4 & $x_1 (x_1^2 + x_2 x_3)$ & union of 2 lines & 11\\ &&&& $[0:0:x_3:x_4]\cup [0:x_2:0:x_4]$\\
7. & 4D & 4 & $x_1 x_2 x_3$ & union of 3 lines & 10 \\
&&&& $[0:0:x_3:x_4]\cup [0:x_2:0:x_4]$\\
&&&& $\cup [x_1:0:0:x_4]$\\
8. & 5A & 5 & $x_2 (x_1^2 + x_2 x_3)$ & line $[0:0:x_3:x_4]$ & 10 \\
9. & 6A & 6 & $x_1 (x_1^2 + x_2^2 + x_3^2 + x_4^2)$ & $x_1=0$, $x_2^2 + x_3^2 + x_4^2 = 0$ & 13 \\
10. & 6B & 6 & $x_1(x_2^2 + x_3^2 + x_4^2)$ & $[1:0:0:0]\cup \{x_2^2 + x_3^2 + x_4^2 = 0\}$ & 12 \\
11. & 6C & 6 & $x_1 x_2^2 + x_3 x_4^2$ & line $[x_1:0:x_3:0]$ & 14 \\
12. & 7A & 7 & $x_1 (x_1 x_2 + x_3^2 + x_4^2)$ & $x_1=0$, $x_3^2 + x_4^2 = 0$ & 12 \\
13. & 7B & 7 & $x_1^2 x_2 + x_1 x_3 x_4 + x_3^3$ & line $[0:x_2:0:x_4]$ & 13
\end{tabular}
\caption{Normal forms of cubic surfaces with infinitely many singular points. Note: "Dim" denotes the dimension of the orbit closure.}
\label{table:infinite_singularities_normal_form}
\end{table}

We give a partial classification of the orbit closure containments for these normal forms in Table \ref{table:infinite_singularities_orbit_closure}. We will use the notation (row, column): a \cmark means that the orbit closure of ``row'' is contained in that of ``column'', while a \xmark indicates otherwise. 
%Quotation marks in LaTeX are weird:  you need to type `` for open quotation marks, and '' for close quotation marks (the second one was two separate apostrophes) -- FIXED

\begin{table}[ht]
\centering
\begin{tabular}{c|c|c|c|c|c|c|c|c|c|c|c|c|c}
Label & 1A & 2A & 3A & 4A & 4B & 4C & 4D & 5A & 6A & 6B & 6C & 7A & 7B \\ \hline
1A & \cmark & \cmark & \cmark & \cmark & \cmark & \cmark & \cmark & \cmark & \cmark & \cmark & \cmark & \cmark & \cmark \\ \hline
2A & \gxmark & \cmark & \gxmark & \cmark & \cmark & \cmark & \cmark & \cmark & \cmark & \cmark & \cmark & \cmark & \cmark \\ \hline
3A & \gxmark & \cmark & \cmark & \cmark & \cmark & \cmark & \cmark & \cmark & \cmark & \cmark & \cmark & \cmark & \cmark \\ \hline
4A & \gxmark & \gxmark & \gxmark & \cmark & \gxmark & \gxmark & \gxmark & \gxmark & & \gxmark & \cmark & \gxmark & \\ \hline
4B & \gxmark & \gxmark & \gxmark & \cmark & \cmark & \gxmark & \gxmark & \gxmark &&& \cmark && \cmark \\ \hline
4C & \gxmark & \gxmark & \gxmark & \cmark & \gxmark & \cmark & \gxmark & \gxmark & \cmark & \cmark & \cmark & \cmark & \cmark \\ \hline
4D & \gxmark & \gxmark & \gxmark & \cmark & & \cmark & \cmark & \gxmark & \cmark & \cmark & \cmark & \cmark & \cmark \\ \hline
5A & \gxmark & \gxmark & \gxmark & \cmark & \cmark & \cmark & \gxmark & \cmark & \cmark & \cmark & \cmark & \cmark & \cmark \\ \hline
6A & \gxmark & \gxmark & \gxmark & \gxmark & \gxmark & \gxmark & \gxmark & \gxmark & \cmark & \gxmark & & \gxmark & \gxmark \\ \hline
6B & \gxmark & \gxmark & \gxmark & \gxmark & \gxmark & \gxmark & \gxmark & \gxmark & \cmark & \cmark & & \gxmark & \\ \hline
6C & \gxmark & \gxmark & \gxmark & \gxmark & \gxmark & \gxmark & \gxmark & \gxmark & \gxmark & \gxmark & \cmark & \gxmark & \gxmark \\ \hline
7A & \gxmark & \gxmark & \gxmark & \gxmark & \gxmark & \gxmark & \gxmark & \gxmark & \cmark & \gxmark & & \cmark & \\ \hline
7B & \gxmark & \gxmark & \gxmark & \gxmark & \gxmark & \gxmark & \gxmark & \gxmark & \gxmark & \gxmark & \cmark & \gxmark & \cmark 
\end{tabular}
\caption{Normal forms of cubic surfaces with infinitely many singular points}
\label{table:infinite_singularities_orbit_closure}
\end{table}

\begin{itemize}
    \item [Justification:]
    A \textcolor{gray}{grey} entry indicates that the result is due to a dimension argument. 
    
    The diagonal entries are trivial.
    \item [Row 1:]
    (1A, 1A), (1A, 2A), (1A, 4A), (1A, 4B), (1A, 4C), (1A, 6A), (1A, 7B) are clear.
    
    (1A, 3A): $x_1^3 = \lim_{\epsilon \to 0} x_1^2(x_1+\epsilon x_2)$. 
    
    (1A, 4D): $x_1^3 = \lim_{\epsilon \to 0} x_1(x_1 + \epsilon x_2)(x_1 + \epsilon x_3)$. 
    
    (3A, 5A), (3A, 6B), (3A, 6C), (3A, 7A) imply (1A, 5A), (1A, 6B), (1A, 6C), (1A, 7A) respectively.
    
    \item [Row 2:]
    (2A, 4B): $x_1^3 + x_2^3 = \lim_{\epsilon \to 0} x_2^2(x_2 + \epsilon x_3) - (-x_1)^3$.
    
    (2A, 4A): (2A, 4B) and (4B, 4A) $\Rightarrow$ (2A, 4A).
    
    (2A, 4C), (2A, 4D), (2A, 5A): Macaulay2 computations.
    \begin{footnotesize}
    \begin{verbatim}
        R = QQ[c3000,c2100,c1200,c0300,c2010,c1110,c0210,c1020,c0120,c0030,
            a1,a2,a3,b1,b2,b3,c1,c2,c3][x1,x2,x3];
        M = matrix {{x1^3, x1^2*x2, x1*x2^2, x2^3, x1^2*x3, x1*x2*x3, x2^2*x3, 
            x1*x3^2, x2*x3^2, x3^3}}
        C = matrix{{c3000,c2100,c1200,c0300,c2010,c1110,c0210,c1020,c0120,c0030}};
        f1 = M*transpose(C);
        f1 = f1_(0,0);
        l1 = a1*x1 + a2*x2 + a3*x3;
        l2 = b1*x1 + b2*x2 + b3*x3;
        l3 = c1*x1 + c2*x2 + c3*x3;
        f2 = l1*(l1^2 + l2*l3); -- l1*l2*l3, l2*(l1^2 + l2*l3) respectively
        g = f1 - f2;
        (M,C) = coefficients g;
        I = minors(1,C);
        S = QQ[c3000,c2100,c1200,c0300,c2010,c1110,c0210,c1020,c0120,c0030,
            a1,a2,a3,b1,b2,b3,c1,c2,c3];
        I = substitute(I,S);
        J = substitute(I,{c3000=>1,c2100=>0,c1200=>0,c0300=>1,c2010=>0,c1110=>0,
            c0210=>0,c1020=>0,c0120=>0,c0030=>0});
        time K = eliminate(J,{a1,a2,a3,b1,b2,b3,c1,c2,c3});
    \end{verbatim}
    \end{footnotesize}
    (2A, 6B), (2A, 7A): We use the same Macaulay2 code as (2A, 4C), replacing f2 by $l_1(l_2^2 + l_3^2),\ l_1(l_1l_2 + l_3^2)$ respectively. Since the former is in the orbit closure of 6B and the latter in that of 7A, we get (2A, 6B), (2A, 7A).
    
    (2A, 6A): (2A, 6B) and (6B, 6A) $\Rightarrow$ (2A, 6A).
    
    (2A, 7B): $x_1^3 + x_2^3 = \lim_{\epsilon \to 0} x_1^2(x_1 + \epsilon x_2) + x_1(x_2 + \epsilon x_3)\epsilon x_4 + (x_2 + \epsilon x_3)^3$.
    
    (2A, 6C): (2A, 7B) and (7B, 6C) $\Rightarrow$ (2A, 6C).
    
    \item [Row 3:]
    (3A, 2A): $\lim_{\epsilon \to 0} \frac{x_1^3 + (-x_1 + \epsilon x_2)^3}{3\epsilon} = x_1^2x_2.$
    
    (3A, 4A), (3A, 4B), (3A, 4C), (3A, 4D), (3A, 5A), (3A, 6A), (3A, 6B), (3A, 6C), (3A, 7A), (3A, 7B): Inherit from 2A.
    
    \item [Row 4:]
    (4A, 6C): $\lim_{\epsilon \to 0}x_3x_2^2 + (-x_1-x_3)(x_1+\epsilon x_4)^2 = x_2^2x_3 - x_1^3 - x_1^2x_3.$
    
    \item [Row 5:]
    (4B, 4A): $\lim_{\epsilon \to 0}(\frac{x_2}{\epsilon})^2(\epsilon^2x_3) - x_1^3 - x_1^2(\epsilon^2 x_3) = x_2^2x_3 - x_1^3.$
    
    (4B, 6C): Inherit from 4A.
    
    (4B, 7B): $\lim_{\epsilon \to 0}x_2^2x_3 + x_2(-x_1)(\epsilon x_4) + (-x_1)^3 = x_2^2x_3 - x_1^3.$
    
    \item [Row 6:]
    (4C, 4A): $\lim_{\epsilon \to 0} -(-x_1)^3 + \frac{x_3}{2\epsilon}({-(-x_1)^2 + (x_1 + \epsilon x_2)^2)} = x_1(x_1^2 + x_2x_3).$
    
    (4C, 6B): $\lim_{\epsilon \to 0}x_1\left(\frac{x_2^2 + (i(x_2-\epsilon x_3))^2}{2\epsilon} + (x_1 + \epsilon x_4)^2 \right) = x_1x_2x_3 + x_1^3.$
    
    (4C, 6A): Inherit from 6B.
    
    (4C, 7A): Let $y_1 = x_1,\ y_2 = x_1 + \epsilon x_4,\ y_3 = \frac{x_2}{\sqrt{2\epsilon}},\ y_4 = \frac{i(x_2 - \epsilon x_3)}{\sqrt{2\epsilon}}$. Then, $\lim_{\epsilon \to 0}y_1(y_1y_2 + y_3^2 + y_4^2) = x_1(x_1^2 + x_2x_3).$
    
    (4C, 7B): $\lim_{\epsilon \to 0} x_1^2(\epsilon x_4) + x_3x_1x_2 + x_1^3 = x_1x_2x_3 + x_1^3.$
    
    (4C, 6C): Inherit from 7B.
    
    \item [Row 7:]
    (4D, 4A): $\lim_{\epsilon \to 0} ((\epsilon(x_1 + \epsilon x_2))^2\frac{x_3}{2\epsilon^3} - (\epsilon x_1)^3 - (\epsilon x_1)^2\frac{x_3}{2\epsilon^3} = x_1x_2x_3.$
    
    (4D, 4C): $\lim_{\epsilon \to 0}\frac{(\epsilon x_1)^3 + \epsilon x_1x_2x_3}{\epsilon} = x_1x_2x_3.$
    
    (4D, 6B): $\lim_{\epsilon \to 0} x_1\left(\frac{x_2^2 + (i(x_2 - \epsilon x_3))^2}{2\epsilon} + \epsilon x_4^2\right) = x_1x_2x_3.$
    
    (4D, 6A): Inherit from 6B.
    
    (4D, 7A): $\lim_{\epsilon \to 0} x_1\left(x_1(\epsilon x_4) + \frac{x_2^2 + (i(x_2 - \epsilon x_3))^2}{2\epsilon}\right) = x_1x_2x_3.$
    
    (4D, 7B): $\lim_{\epsilon \to 0} \frac{x_1^2(\epsilon^2 x_4) + x_1(\epsilon x_3)x_2 + (\epsilon x_3)^3}{\epsilon} = x_1x_2x_3.$
    
    (4D, 6C): Inherit from 7B.
    
    \item [Row 8:]
    (5A, 4B): Macaulay2 computation.
    \begin{footnotesize}
    \begin{verbatim}
        R = QQ[c3000,c2100,c1200,c0300,c2010,c1110,c0210,c1020,c0120,c0030,
               a1,a2,a3,b1,b2,b3,c1,c2,c3][x1,x2,x3];
        M = matrix {{x1^3, x1^2*x2, x1*x2^2, x2^3, x1^2*x3, x1*x2*x3, x2^2*x3, x1*x3^2, 
               x2*x3^2, x3^3}}
        C = matrix{{c3000,c2100,c1200,c0300,c2010,c1110,c0210,c1020,c0120,c0030}};
        f1 = M*transpose(C);
        f1 = f1_(0,0);
        l1 = a1*x1 + a2*x2 + a3*x3;
        l2 = b1*x1 + b2*x2 + b3*x3;
        l3 = c1*x1 + c2*x2 + c3*x3;
        f2 = l2^2*l3 - l1^3
        g = f1 - f2;
        (M,C) = coefficients g;
        I = minors(1,C);
        S = QQ[c3000,c2100,c1200,c0300,c2010,c1110,c0210,c1020,c0120,c0030,
               a1,a2,a3,b1,b2,b3,c1,c2,c3];
        I = substitute(I,S);
        J = substitute(I,{c3000=>0,c1200=>0,c0300=>0,c2010=>0});
        time K = eliminate(J,{a1,a2,a3,b1,b2,b3,c1,c2,c3});
        K = substitute(K,{c2100=>1,c0210=>1,c1110=>0,c1020=>0,c0120=>0,c0030=>0})
    \end{verbatim}
    \end{footnotesize}
    
    (5A, 4A): Inherit from 4B.
    
    (5A, 4C): Macaulay2 computation.
    \begin{footnotesize}
    \begin{verbatim}
        R = QQ[c3000,c2100,c1200,c0300,c2010,c1110,c0210,c1020,c0120,c0030,
               a1,a2,a3,b1,b2,b3,c1,c2,c3][x1,x2,x3];
        M = matrix {{x1^3, x1^2*x2, x1*x2^2, x2^3, x1^2*x3, x1*x2*x3, x2^2*x3, 
               x1*x3^2, x2*x3^2, x3^3}}
        C = matrix{{c3000,c2100,c1200,c0300,c2010,c1110,c0210,c1020,c0120,c0030}};
        f1 = M*transpose(C);
        f1 = f1_(0,0);
        l1 = a1*x1 + a2*x2 + a3*x3;
        l2 = b1*x1 + b2*x2 + b3*x3;
        l3 = c1*x1 + c2*x2 + c3*x3;
        f2 = l1^3 + l1*l2*l3
        g = f1 - f2;
        (M,C) = coefficients g;
        I = minors(1,C);
        S = QQ[c3000,c2100,c1200,c0300,c2010,c1110,c0210,c1020,c0120,c0030,
               a1,a2,a3,b1,b2,b3,c1,c2,c3];
        I = substitute(I,S);
        J = substitute(I,{c3000=>0,c1200=>0,c0300=>0,c2010=>0,c1110=>0,c1020=>0});
        time K = eliminate(J,{a1,a2,a3,b1,b2,b3,c1,c2,c3});
        K = substitute(K,{c2100=>1,c0210=>1,c0120=>0,c0030=>0})
    \end{verbatim}
    \end{footnotesize}
    
    (5A, 6B): $\lim_{\epsilon \to 0} x_2\left(x_1^2 + \frac{x_3^2 + (i(x_3 - \epsilon x_2))^2}{2\epsilon}\right) = x_2(x_1^2 + x_2x_3).$
    
    (5A, 6A): Inherit from 6B.
    
    (5A, 7A): $\lim_{\epsilon \to 0} x_2(x_2x_3 + x_1^2 + (\epsilon x_4)^2) = x_2(x_1^2 + x_2x_3).$
    
    (5A, 7B): Inherit from 4B.
    
    (5A, 6C): Inherit from 7B.
    
    \item [Row 9:]
    (6B, 6A): $\lim_{\epsilon \to 0} \frac{\epsilon x_1((\epsilon x_1)^2 + x_2^2 + x_3^2 + x_4^2)}{\epsilon} = x_1(x_2^2 + x_3^2 + x_4^2).$
    
    \item [Row 10:]
    \item [Row 12:]
    (7A, 6A): Let $y_1 = \frac{x_1}{\sqrt{2}\sqrt[3]{\epsilon}},\ y_2 = \frac{i(x_1 - \epsilon x_2)}{\sqrt{2}\sqrt[3]{\epsilon}},\ y_3 = \sqrt[4]{2}\sqrt[6]{\epsilon}x_3,\ y_4 = \sqrt[4]{2}\sqrt[6]{\epsilon}x_4$. 
    Then $\lim_{\epsilon \to 0}y_1(y_1^2 + y_2^2 + y_3^2 + y_4^2) = x_1(x_1x_2 + x_3^2 + x_4^2).$
    
    \item [Row 13:]
    (7B, 6C): This is Example \ref{exmp:sub-elim-sub-example}.
\end{itemize}

There are still 11 ordered pairs (left blank) for which we are unable to compute the orbit closure containment at this time. For these cases, either there are four variables involved causing elimination to not terminate, or the dimensions of the orbits of the two cubic surfaces considered differ only by one, so the guessing in \emph{sub-elim-sub} is hard. 

Nevertheless, we obtain a stronger statement of~\cite[Corollary 3.4]{AE}.

\medskip
\noindent \textbf{``Corollary 3.4''} A general cubic surface with infinitely many singular points has rank six.

\begin{proof}
There are finitely many orbits of cubic surfaces with infinitely many singular points. All the normal forms of cubic surfaces with infinitely many singular points are contained in the orbit closures of either $x_1(x_1^2 + x_2^2 + x_3^2 + x_4^2)$ or $x_1x_2^2 + x_3x_4^2$, which have rank six. 
\end{proof}

Out of the 11 cases for which we are unable to decide the orbit closure containments, only 1 case involves 3 variables while the other 10 involve 4 variables. As mentioned in Remark \ref{rmk:justification-m2}, the Macaulay2 implementation of \emph{eliminate} involves computing a Gröbner basis, which is done using Buchberger's algorithm. Buchberger's algorithm is slow in general, especially when the number of variables is large. 

We would also like to note that in~\cite[Theorem 2.12]{Popov}, the author provides an algorithm which reduces the orbit closure problem to checking the consistency of a system of linear equations. With notation as in~\cite{Popov}, for the case of cubic surfaces, after reduction to the conic case~\cite[2.7]{Popov}, we have $V = $ the 20-dimensional vector space of cubic surfaces, $G = GL(4, \C)$, and $\rho: G \to \mbox{Mat}_{20,20}(\C)$ is the matrix representation of $G$, i.e. $\rho(g)$ is the matrix of the linear transformation $V\to V$, $v\mapsto g\cdot v$ with respect to the monomial basis $\{x_1^ix_2^jx_3^kx_4^l \vert i+j+k+l = 3\}$.

The algorithm takes as input $d = \deg \rho (G)$. According to an (unpublished) computation done by Hanieh Keneshlou and Khazhgali Kozhasov, this degree is very large: 4306472. Since step 2 of the algorithm~\cite[2.11]{Popov} involves a generic polynomial of degree $2d-2$, for the case of cubic surfaces, it does not make the orbit closure problem easier.

As Bernd Sturmfels suggested, another approach for future projects to tackle this problem would be numerical algebraic geometry (see e.g.~\cite{NumAG} or~\cite{Bertini}).  This approach has been recently used to great success in such papers as~\cite{3264}.

\bigskip
\noindent \textbf{Acknowledgements} 
We would like to thank Ralph Morrison for helpful guidance and suggestions throughout the project. We would also like to thank Anna Seigal for helpful conversation and for providing the codes for orbit closure and stabilizer from a previous project. We are also thankful to Hanieh Keneshlou and Khazhgali Kozhasov for computing the degree 4306472. Last but not least, we are really grateful to Bernd Sturmfels for helpful discussions and for the opportunity to work on this project.

\end{document}